\documentclass[a4paper]{amsart}
        \usepackage{latexsym}
        \usepackage{amssymb}
        \usepackage{amsmath}
        \usepackage{amsfonts}
        \usepackage{amsthm}
        \usepackage[hypertexnames=false]{hyperref}
        \ifpdf
          \usepackage[all,knot,poly,2cell]{xy}
        \else
          \usepackage[all,knot,poly,2cell,dvips]{xy}
        \fi
             \UseAllTwocells
        \usepackage{xspace}

        \usepackage{eucal}
        \theoremstyle{plain}
        \newtheorem{theorem}{Theorem}[section]
        \newtheorem{corollary}[theorem]{Corollary}
        \newtheorem{lemma}[theorem]{Lemma}

        \theoremstyle{definition}
        \newtheorem{definition}[theorem]{Definition}
        \newtheorem{example}[theorem]{Example}

        \theoremstyle{remark}
        
        \newcommand{\suchthat}{\,:\,}
        \newcommand{\itemref}[1]{\eqref{#1}}

        \newcommand{\Z}{\mathbb{Z}}
        
        \newcommand{\F}{\mathbb{F}}

        \newcommand{\Orb}{\mathcal{O}}   % structure sheaf 
       \DeclareMathOperator{\spec}{Spec} % spectrum of a ring
        \newcommand{\red}[1]{{#1}_{\mathrm{red}}} % reduction of a space
           \newcommand{\lisset}{\mathrm{lis\text{\nobreakdash-}\acute{e}t}}

        \newcommand{\MOD}{\mathsf{Mod}}    % the category of modules      

        \DeclareMathOperator{\coker}{coker}
         \DeclareMathOperator{\Hom}{Hom}

        \DeclareMathOperator{\supp}{supp}
        \newcommand{\COHO}[1]{\mathcal{H}^{{#1}}}
        \newcommand{\trunc}[1]{\tau^{{#1}}}
        \newcommand{\RDERF}{\mathsf{R}}
        \newcommand{\LDERF}{\mathsf{L}}
        
        \newcommand{\DCAT}{\mathsf{D}}

        \DeclareMathOperator{\supph}{supph} % homological support
        \newcommand{\QCOH}{\mathsf{QCoh}}
        
        \newcommand{\PERF}{\mathsf{Perf}}

        \newcommand{\ID}[1]{\mathrm{Id}_{#1}}    
        
        \DeclareMathOperator{\End}{End}

        \newcommand{\tensor}{\otimes}

\renewcommand{\mid}{|}

\numberwithin{equation}{section}

\newcommand{\qcsubscript}{\mathrm{qc}} % subscript

\newcommand{\DQCOH}[1][]{\DCAT_{\qcsubscript{#1}}} 

\newcommand{\QCPSH}[1]{(#1_{\qcsubscript})_*} %% this is for D_qc
\newcommand{\QCPBK}[1]{#1^*_{\qcsubscript}}

\newcommand{\spref}[1]{\href{http://stacks.math.columbia.edu/tag/#1}{#1}}

\makeatletter
\newcommand{\labitem}[2]{%
\def\@itemlabel{(\textbf{#1})}
\item
\def\@currentlabel{\textbf{#1}}\label{#2}}
\makeatother
\newcommand{\cms}{\mathrm{cms}}
\newcommand{\SPB}{\mathrm{Sp}_{\mathrm{Bal}}}
\newcommand{\tensthck}[1]{\langle{#1}\rangle_{\tensor}}

\title{The Balmer spectrum of a tame stack}
\date{November 24, 2014}
\author[J. Hall]{Jack Hall}
\address{Mathematical Sciences Institute\\The Australian National
  University\\Acton, ACT, 2601\\Australia}
\email{jack.hall@anu.edu.au}


\renewcommand{\SPB}{\mathrm{Sp}^{\mathrm{Bal}}}
\renewcommand{\cms}{\mathrm{cs}}
\newcommand{\STACK}{\mathbf{Stack}}
\begin{document}
\maketitle
\begin{abstract}
  Let $X$ be a quasi-compact algebraic stack with quasi-finite and separated diagonal. We classify the thick $\otimes$-ideals of $\DQCOH(X)^c$. If $X$ is tame, then we also compute the Balmer spectrum of the $\otimes$-triangulated category of perfect complexes on $X$. In addition, if $X$ admits a coarse space $X_{\mathrm{cs}}$, then we prove that the Balmer spectra of $X$ and $X_{\mathrm{cs}}$ are naturally isomorphic. 
\end{abstract}

\section{Introduction}
Let $X$ be a quasi-compact and quasi-separated scheme. Let $\PERF(X)$
be the $\tensor$-triangulated category of perfect complexes on $X$. A
celebrated result of Thomason \cite[Thm.~3.15]{MR1436741}, extending
the work of Hopkins \cite[\S 4]{MR932260} and Neeman
\cite[Thm.~1.5]{MR1174255}, is a classification of the thick
$\tensor$-ideals of $\PERF(X)$ in terms of the \emph{Thomason} subsets
of $|X|$, which are those subsets $Y \subseteq |X|$ expressible as a
union $\cup_\alpha Y_\alpha$ such that $|X|\setminus Y_\alpha$ is
quasi-compact and open.

If $X$ is a quasi-compact and quasi-separated algebraic space,
Deligne--Mumford stack, or algebraic stack, then it is also natural to
consider the $\tensor$-triangulated category $\PERF(X)$ of perfect
complexes on $X$ (see \cite[\S3]{perfect_complexes_stacks} for precise
definitions). 

In general, Thomason's classification of thick $\tensor$-ideals of
$\PERF(X)$ fails for algebraic stacks (Example \ref{E:nontame}). If
one instead works with the $\tensor$-ideal $\DQCOH(X)^c \subseteq
\PERF(X)$ of \emph{compact} perfect complexes, then the first main result of
this article is that the classification goes through without change.
\begin{theorem}[Classification of thick
  $\tensor$-ideals]\label{T:thomason_classification}
  Let $X$ be a quasi-compact algebraic stack with quasi-finite and
  separated diagonal. Then there is a bijective correspondence
  between thick $\tensor$-ideals of $\DQCOH(X)^c$ and Thomason subsets of
  $|X|$.
\end{theorem}
Some special cases of Theorem \ref{T:thomason_classification} are
the following:
\begin{itemize}
\item if $k$ is a field and $G$ is a finite group, then
  $\DCAT^b(\mathsf{Proj}\,kG)$ has no non-trivial $\tensor$-ideals;
\item if $Y$ is a quasi-projective
  scheme over a field $k$ with a proper action of a group scheme $G$,
  then the thick $\tensor$-ideals of $\DCAT(\QCOH^G(Y))^c$ are in
  bijective correspondence with the $G$-invariant Thomason subsets of
  $X$. 
\end{itemize}
The first special case is easy to prove directly and is well-known
(cf.~\cite[Prop.~2.1]{MR2846489}). In some sense, this makes our
results orthogonal to \cite{MR2846489}. The second special case was only known
in characteristic $0$ when $Y$ was normal or quasi-affine
\cite[Thm.~7.8]{MR2570954} or in characteristic $p$ when $G$ is finite
of order prime to $p$ and $X$ is smooth \cite[Thm.~1.2]{MR2927050}.

We prove Theorem \ref{T:thomason_classification} using tensor
nilpotence with parameters (Theorem \ref{T:tens_nilp}), which extends \cite[Thm.~3.8]{MR1436741} and
\cite[Thm.~10ii]{MR932260} (cf. \cite[1.1]{MR1174255}) to quasi-compact algebraic
stacks with quasi-finite and separated diagonal. As should be expected, stacks of the form $[Y/G]$, where $Y$ is an
affine variety over a field $k$ and $G$ is a finite group with order divisible by
the characteristic of $k$, are the most troublesome. This is dealt
with in Lemma \ref{L:finite_splitting}, which relies on some results
developed in Appendix \ref{A:tamestacks}.

If $\mathcal{T}$ is a $\tensor$-triangulated category, then Balmer
\cite{MR2196732} has functorially constructed from $\mathcal{T}$ a
locally ringed space $\SPB(\mathcal{T})$, the \emph{Balmer spectrum}. 
A fundamental result of Balmer \cite[Thm.~5.5]{MR2196732}, which was
extended by Buan--Krause--Solberg \cite[Thm.~9.5]{MR2280286} to the
non-noetherian setting, is that if $X$ is a quasi-compact and
quasi-separated scheme, then there is a naturally induced isomorphism
\[
X \to \SPB(\PERF(X)).
\]

An algebraic stack is \emph{tame} if its stabilizer groups at geometric points are 
finite linearly reductive group schemes \cite[Defn.~2.2]{MR2427954}. Every scheme and algebraic space is
tame. Moreover, in characteristic zero, a stack is Deligne--Mumford if and
only if is tame. In characteristic $p>0$, there are non-tame
Deligne--Mumford stacks (e.g., $B_{\F_p}(\Z/p\Z)$) and tame stacks
that are not Deligne--Mumford (e.g., $B_{\F_p}\mu_p$). Nagata's
Theorem \cite[Thm.~1.2]{hallj_dary_alg_groups_classifying} provides a
classification of finite linearly reductive group schemes over fields, thus 
allows one to determine whether a given algebraic stack is tame. Note that our definition of tame stack is substantially weaker than that what appears in \cite[Defn.~3.1]{MR2427954} (see Appendix \ref{A:tamestacks}).

Tame stacks are precisely those stacks with quasi-finite diagonal such
that the compact objects of $\DQCOH(X)$ coincide with the perfect complexes. Using Theorem
\ref{T:thomason_classification}, we extend the result of
Buan--Krause--Solberg to tame stacks. 
\begin{theorem}\label{T:balmer_tame}
  Let $X$ be a quasi-compact algebraic stack with quasi-finite and separated diagonal. If $X$ is tame, then there is a natural isomorphism of locally ringed spaces
  \[
  (|X|,\Orb_{X_{\mathrm{Zar}}}) \to \SPB(\PERF(X)),
  \]
  where $\Orb_{X_{\mathrm{Zar}}}$ is the Zariski sheaf $U\mapsto \Gamma(U,\Orb_X)$.
\end{theorem}
Theorem \ref{T:balmer_tame} implies that the Balmer spectrum cannot be used to reconstruct locally separated algebraic spaces \cite[Ex.~2]{MR0302647}. Balmer \cite{2013arXiv1309.1808B} has recently initiated the study of unramified monoids in $\tensor$-triangulated categories and Neeman \cite{neeman_unramified2014} has classified them in the case of a separated noetherian scheme. It is hoped that a refinement of the Balmer spectrum can be constructed from unramified monoids, which would---at least---permit the reconstruction of algebraic spaces. 
 
If $X$ is an algebraic stack with finite inertia (e.g. a separated
Deligne--Mumford stack), then $X$ admits a \emph{coarse space} $\pi
\colon X \to X_{\cms}$ \cite{MR1432041,MR3084720}, which is the
universal map from $X$ to an algebraic space. If $X$ has finite inertia, then $X$ has separated diagonal. Thus we can also establish the following.
\begin{theorem}\label{T:balmer_tame_coarse}
  Let $X$ be a quasi-compact, quasi-separated algebraic stack with
  finite inertia and coarse space $\pi\colon X \to X_{\cms}$. If $X$
  is tame, then
  \[
  \SPB(\LDERF \pi^*) \colon \SPB(\PERF(X)) \to \SPB(\PERF(X_{\cms}))
  \]
  is an isomorphism of ringed spaces. 
\end{theorem}
Krishna \cite[Thm.~7.10]{MR2570954} proved Theorem
\ref{T:balmer_tame_coarse} when $X$ is of the form $[W/G]$, where $W$
is quasi-projective and normal or quasi-affine, and $G$ is a linear
algebraic group in characteristic $0$ acting properly on
$W$. Dubey--Mallick \cite[Thm.~1.2]{MR2927050} proved a similar result
in positive characteristic, but required $W$ to be smooth and $G$ a
finite group with order not divisible by the characteristic of the
ground field. In particular, Theorem \ref{T:balmer_tame_coarse} is
stronger than all existing results and Theorems \ref{T:thomason_classification} and \ref{T:balmer_tame} are new.
\subsection*{Assumptions and conventions}
A priori, we make no separation assumptions on our algebraic
stacks. However, all stacks used in this article will be, at the
least, quasi-compact and quasi-separated. Usually, they will also have separated diagonal. If $X$ is an algebraic stack, then let $|X|$ denote its associated Zariski topological space \cite[\S5]{MR1771927}. For derived categories of
algebraic stacks, we use the conventions and notations of \cite[\S
1]{perfect_complexes_stacks}. In particular, if $X$ is an algebraic
stack, then $\MOD(X)$ is the abelian category of $\Orb_X$-modules on
the lisse-\'etale site of $X$ and $\DQCOH(X)$ denotes the unbounded
derived category of $\Orb_{X_{\lisset}}$-modules with quasi-coherent
cohomology sheaves. If $f\colon X \to Y$ is a morphism of algebraic
stacks, then there is always an adjoint pair of unbounded derived
functors
\[
\xymatrix{\DQCOH(X) \ar@<0.5ex>[r]^{\RDERF \QCPSH{f}} &
  \ar@<0.5ex>[l]^{\LDERF \QCPBK{f}} \DQCOH(Y).}
\]
If $f$ is quasi-compact, quasi-separated and representable, then
$\RDERF \QCPSH{f}$ agrees with $\RDERF f_*$, the unbounded derived
functor of $f_* \colon \MOD(X) \to \MOD(Y)$ \cite[Lem.~2.5(3)\ \&\
Thm.~2.6(2)]{perfect_complexes_stacks}. If $f$ is flat, then
$\LDERF \QCPBK{f}$ agrees with the unique extension of the exact
functor $f^* \colon \MOD(Y) \to \MOD(X)$ to the unbounded derived
category.
\subsection*{Acknowledgements}
I would like to thank Amnon Neeman for his encouragement and several
useful discussions and Ben Antieau for some encouraging remarks and observations. I would also like to thank David Rydh for several
useful suggestions regarding tame stacks and their coarse moduli spaces. 
\section{Tensor nilpotence with parameters}
We begin with the following definition.
\begin{definition}
  Let $X$ be an algebraic stack and let $\xi \colon M \to N$ be a morphism in
  $\DQCOH(X)$. If $Z \subseteq |X|$ is a subset, then $\xi$
  \emph{vanishes at the points of $Z$} if for every algebraically
  closed field $k$ and morphism $z\colon \spec k \to X$ that factors
  through $Z$, then $\LDERF \QCPBK{z}\xi$ is the zero map in
  $\DQCOH(\spec k)$. 
\end{definition}
The following lemma will connect this definition with a more
familiar notion for schemes. 
\begin{lemma}\label{L:vanishing_wd}
  Let $X$ be a scheme and let $\xi \colon M \to N$ be a morphism in
  $\DQCOH(X)$. If $Z \subseteq |X|$ is a subset, then $\xi$
  \emph{vanishes at the points of $Z$} if and only if
  $\xi\tensor^{\LDERF}_{\Orb_X} \kappa(z)$ is the zero map in
  $\DCAT(\kappa(z))$ for every $z \in Z$, where $\kappa(z)$ denotes
  the residue field of $z$. 
\end{lemma}
\begin{proof}
  We immediately reduce to the situation where $X=\spec \kappa$ and $\kappa$
  is a field. It now suffices to prove that if $\kappa \subseteq k$ is a
  field extension, where $k$ is algebraically closed, then $\xi\tensor
  k$ is the zero map in $\DCAT(k)$ if and only if $\xi$ is the zero
  map in $\DCAT(\kappa)$. This is obvious. 
\end{proof}
If $K \in \DQCOH(X)$, then the
\emph{cohomological support} of $K$ is defined to be the subset
\[
\supph(K) = \cup_{n\in \Z} \supp(\COHO{n}(K)) \subseteq |X|.
\]
For the basic properties of cohomological support, see
\cite[Lem.~3.5]{perfect_complexes_stacks}, which extends
\cite[Lem.~3.3]{MR1436741} to algebraic stacks. The main result of
this section is the following theorem.
\begin{theorem}[Tensor nilpotence with parameters]\label{T:tens_nilp}
  Let $X$ be a quasi-compact algebaic stack with quasi-finite and
  separated diagonal. Let $\phi \colon E \to F$ be a morphism in
  $\DQCOH(X)$, where $E \in \DQCOH(X)^c$. Let $K\in \PERF(X)$. If
  $\phi$ vanishes at the points of $\supph(K)$, then there exists a positive integer $n$ such that $K\tensor^{\LDERF}_{\Orb_X}(\psi^{\tensor n})=0$ in $\DQCOH(X)$. 
\end{theorem}
The following example demonstrates that Theorem \ref{T:tens_nilp}
cannot be weakened to the situation where $E \in \PERF(X)$.
\begin{example}\label{E:tensor_nontame}
  Let $X=B_{\F_2}(\Z/2\Z)$. That $X$ is a quasi-compact, non-tame
  Deligne--Mumford stack with finite diagonal. Consider the adjunction morphism
  $\eta\colon \Orb_X
  \to x_*\Orb_{\F_2}$, where $x \colon \spec \F_2 \to X$ is the usual
  cover. Then there is a natural map $\phi\colon C \to \Orb_X[1]$, where
  $C=\coker(\eta)$. Clearly, $\phi$ vanishes at the points of $|X|$ (because $x^*\eta$ is
  split). If $\phi^{\tensor n}=0$ for some $n$, then it is easily
  determined that this implies that $\Orb_X \in \DQCOH(X)^c$, which is
  false. 
\end{example}

\begin{proof}[Proof of Theorem \ref{T:tens_nilp}]
    Let $\mathbf{E}$ be the category of representable, quasi-finite,
  flat and separated morphisms of finite presentation over
  $X$. Let $\mathbf{D} \subseteq \mathbf{E}$ be the full subcategory
  with objects those $(U \xrightarrow{p} X)$ such that there exists an
  integer $n>0$ with $p^*(K \tensor_{\Orb_X}^{\LDERF} (\psi^{\tensor n})) = 0$. It suffices to prove that 
  $\mathbf{D} = \mathbf{E}$. By the induction
  principle (Theorem \ref{T:induction_qf_diag}), it is sufficient to
  verify the following three conditions:
  \begin{enumerate}
  \item[(I1)] if $(U \to W) \in \mathbf{E}$ is an open immersion and $W \in
    \mathbf{D}$, then $U \in \mathbf{D}$;
  \item[(I2)] if $(V \to W) \in \mathbf{E}$ is finite and surjective,
    where $V$ is an affine scheme, then $W \in \mathbf{D}$; and
  \item[(I3)] if $(U \xrightarrow{i} W)$, $(W'\xrightarrow{f} W) \in
    \mathbf{E}$, where $i$ is an open immersion and $f$ is \'etale and
    an isomorphism over $W\setminus U$, then $W \in \mathbf{D}$
    whenever $U$, $W'\in\mathbf{D}$.
  \end{enumerate}
  Now condition (I1) is trivial and condition (I3) is Lemma
  \ref{L:tensor_nilp_local}. For condition (I2), by Lemma
  \ref{L:finite_splitting}, it
  remains to prove that every affine scheme belongs to $\mathbf{D}$.
  By Lemma \ref{L:vanishing_wd} and \cite[Lem.~3.14]{MR1436741} (or \cite[Lem.~1.2]{MR1174255}), the
  result follows.
\end{proof}
\begin{lemma}\label{L:tensor_nilp_local}
  Consider a $2$-cartesian diagram of algebraic stacks
  \[
  \xymatrix{U' \ar@{^(->}[r]^{i'} \ar[d]_{f_U} & W' \ar[d]^f \\ U \ar@{^(->}[r]^i & W,}
  \]
  where $W$ is quasi-compact and quasi-separated, $i$ is a
  quasi-compact open immersion and $f$ is representable, \'etale, finitely presented and
an isomorphism over $X\setminus U$.
  Let $\psi \colon E \to F$ be a morphism in
  $\DQCOH(X)$ and let $K \in \DQCOH(X)$. For each integer $n> 0$, let
  $\phi_n = K \tensor^{\LDERF}_{\Orb_X} (\psi^{\tensor n})$. If $f^*\phi_n = 0$ and $i^*\phi_n = 0$, then $\phi_{2n} = 0$.
\end{lemma}
\begin{proof}
  To simplify notation, we let $E_n = K\tensor_{\Orb_X}^{\LDERF}
  E^{\tensor n}$ and $F_n = K\tensor_{\Orb_X}^{\LDERF} F^{\tensor n}$. We will argue similarly to \cite[Thm.~3.6]{MR1436741}, but
  using the Mayer--Vietoris triangle for \'etale neighbourhoods of
  stacks developed in \cite[Lem.~5.6(1)]{perfect_complexes_stacks}
  instead of \cite[Lem.~3.5]{MR1436741}. 

  Let $k=f\circ i'$. By \cite[Lem.~5.6(1)]{perfect_complexes_stacks}, there is a
  distinguished triangle in $\DQCOH(X)$:
  \[
  \xymatrix{F_n \ar[r] & \RDERF i_*i^*F_n \oplus \RDERF f_*f^*F_n
    \ar[r] & \RDERF k_* k^* F_n \ar[r]^-{d} & F_n[1]. }
  \]
  By applying the homological functor $\Hom_{\Orb_X}(E_n, -)$ to the
  distinguished triangle above, we find that there exists a morphism
  $t \colon E_n \to \RDERF k_* k^* F_n[-1]$ such that $\delta(t) =
  \phi_n$, where $\delta$ is the boundary map induced by $d$. But
  there is a commutative diagram
  \[
  \xymatrix{ & (\RDERF k_* k^* F_n[-1]) \tensor_{\Orb_X}^{\LDERF} E^{\tensor n} \ar[dd]_(0.3){\ID{}\tensor \psi^{\tensor n}} \ar[dr]^{\delta\tensor \psi^{\tensor n}}  & \\
    E_n \tensor_{\Orb_X}^{\LDERF} E^{\tensor n} \ar[ur]^{t\tensor
      \ID{}} \ar[dr]^{t\tensor \psi^{\tensor n}}
    \ar[rr]^(.3){\phi_{2n}} & & F_n \tensor_{\Orb_X}^{\LDERF} F^{\tensor
      n}\\ & (\RDERF k_* k^* F_n[-1])
    \tensor_{\Orb_X}^{\LDERF} F^{\tensor n} \ar[ur]_{\delta\tensor
      \ID{}}, & }
  \]
  so it remains to prove that the vertical map above is zero. To see this, the
  projection formula \cite[Cor.~4.18]{perfect_complexes_stacks}
  implies that we have a commutative diagram
  \[
  \xymatrix{(\RDERF k_* k^* F_n[-1]) \tensor_{\Orb_X}^{\LDERF}
    E^{\tensor n} \ar[r]^-{\sim} \ar[d]_{\ID{}\tensor \psi^{\tensor
        n}} & \RDERF k_*k^*(K \tensor_{\Orb_X}^{\LDERF} F^{\tensor n}
    \tensor_{\Orb_X}^{\LDERF} E^{\tensor n}[-1]) \ar[d]_{\RDERF
      k_*k^*(F^{\tensor n} \tensor \phi_n[-1])} \\ (\RDERF
    k_* k^* F_n[-1]) \tensor_{\Orb_X}^{\LDERF} F^{\tensor
      n} \ar[r]^-{\sim} & \RDERF k_*k^*(K
    \tensor_{\Orb_X}^{\LDERF} F^{\tensor n} \tensor_{\Orb_X}^{\LDERF}
    F^{\tensor n}[-1]). }
  \]
  Since $k^*\phi_n = 0$, the result follows.
\end{proof}
The following lemma is similar to a special case of \cite[Thm.~7.3 \&
Cor.~9.6]{MR2830235}. Also, see \cite[Proof of Prop.~7.6]{MR2570954} and \cite[Lem.~3.8]{MR2927050}. 
\begin{lemma}\label{L:finite_splitting}
  Let $W$ be an algebraic stack and let $v\colon V \to W$ be a finite
  and faithfully flat morphism of finite presentation, where $V$ is an
  affine scheme. Let $\phi \colon E \to F$ be a morphism in
  $\DQCOH(W)$, where $E \in \DQCOH(W)^c$. Let $K \in \PERF(W)$. If
  $v^*(K\tensor_{\Orb_W}^\LDERF \psi) = 0$ in $\DQCOH(V)$, then $K
  \tensor_{\Orb_X}^{\LDERF} \psi = 0$ in $\DQCOH(W)$. 
\end{lemma}
\begin{proof}
  By \cite[Prop.~3.3]{perfect_complexes_stacks}, $\RDERF \QCPSH{v}$
  admits a right adjoint $v^\times$ and there is a functorial isomorphism $v^\times(\Orb_W)
  \tensor_{\Orb_V}^{\LDERF} \LDERF \QCPBK{v}(M) \simeq v^\times(M)$
  for every $M \in \DQCOH(W)$. In particular, if
  $v^*(K\tensor_{\Orb_W}^\LDERF \psi) = 0$ in $\DQCOH(V)$, then
  $v^\times(K \tensor_{\Orb_W}^{\LDERF} \psi) = 0$ in $\DQCOH(V)$. By
  adjunction, it follows that the induced composition
  \[
  \RDERF \QCPSH{v}v^\times(K \tensor_{\Orb_W}^{\LDERF} E) \to K
  \tensor_{\Orb_W}^{\LDERF} E \to K \tensor_{\Orb_W}^\LDERF F
  \]
  vanishes in $\DQCOH(W)$. Thus it suffices to prove that $  \RDERF
  \QCPSH{v} v^\times(K \tensor_{\Orb_W}^{\LDERF} E) \to K
  \tensor_{\Orb_W}^{\LDERF} E$ admits a section. Since $E \in
  \DQCOH(W)^c$ and $K \in \PERF(W)$, it follows that $K
  \tensor_{\Orb_W}^{\LDERF} E \in \DQCOH(W)^c$. Hence, we need only
  prove that if $M \in \DQCOH(W)^c$, then the trace morphism
  $\mathrm{Tr}_M \colon \RDERF
  \QCPSH{v} v^\times(M) \to M$ admits a section. By Lemma
  \ref{L:compactbg}, $M$ is quasi-isomorphic to a direct summand of
  $\RDERF\QCPSH{v}P$, where $P \in \PERF(V)$. Thus we are reduced to
  proving that $\mathrm{Tr}_{\RDERF\QCPSH{v}P}$ admits a section.
  This is trivial and the result follows.
\end{proof}
\section{The classification of thick $\tensor$-ideals}
If $\mathcal{T}$ is a $\tensor$-triangulated category and $S \subseteq
\mathcal{T}$ is a subset, then define $\tensthck{S} \subseteq
\mathcal{T}$ to be the smallest thick $\tensor$-ideal of $\mathcal{T}$
containing $S$. 

In order to prove Theorem \ref{T:thomason_classification}, we require the following lemma, which is analagous to \cite[Lem.~3.14]{MR1436741}.
\begin{lemma}\label{L:supports}
  Let $X$ be a quasi-compact algebraic stack with quasi-finite and separated diagonal. If $P$, $Q \in \DQCOH(X)^c$ and
  $\supph(P) \subseteq \supph(Q)$, then $\tensthck{P} \subseteq
  \tensthck{Q}$.
\end{lemma}
\begin{proof}
  Argue exactly as in \cite[Lem.~3.14]{MR1436741}
  (cf. \cite[Lem.~1.2]{MR1174255}), but this time using Theorem \ref{T:tens_nilp} instead of \cite[Thm.~3.8]{MR1436741}.
\end{proof}
The following example shows Lemma \ref{L:supports} cannot be extended
to $P$, $Q \in \PERF(X)$ when $X$ is non-tame. It also shows that
Thomason's Classification (Theorem \ref{T:thomason_classification})
does not hold for $\PERF(X)$ in this case too.
\begin{example}\label{E:nontame}
  Let $x\colon \spec \F_2 \to X$ be as in Example
  \ref{E:tensor_nontame}. Let $P = \Orb_X$ and
  let $Q = x_*\Orb_{\spec \F_2}$. Then $P$, $Q \in \PERF(X)$ and $\supph(P) =
  \supph(Q)$. Note that $Q \in \DQCOH(X)^c$ and $P \notin
  \DQCOH(X)^c$. Since $\DQCOH(X)^c$ is a thick $\tensor$-ideal of
  $\PERF(X)$, it follows that $\tensthck{Q} \subseteq
  \DQCOH(X)^c$. But if $\tensthck{P} = \tensthck{Q}$, then $P \in
  \DQCOH(X)^c$. But $P\notin \DQCOH(X)^c$, thus we have a contradiction.
\end{example}
Following Thomason \cite[Thm.~3.15]{MR1436741} (or Neeman \cite[Thm.~1.5]{MR1174255}), given Lemma \ref{L:supports}, we can prove Theorem \ref{T:thomason_classification}.
\begin{proof}[Proof of Theorem \ref{T:thomason_classification}]
  If $Y \subseteq |X|$ is a Thomason subset, then define
  \[
  \mathcal{I}_Y = \{ P \in \DQCOH(X)^c \suchthat \supph(P) \subseteq Y\}.
  \]
  Clearly, $\mathcal{I}_Y$ is a thick $\tensor$-ideal of $\DQCOH(X)^c$. If
  $\mathcal{T}$ is a thick $\tensor$-ideal of $\DQCOH(X)^c$, then define
  \[
  \varphi(\mathcal{T}) = \cup_{Q \in \mathcal{T}} \supph(Q).
  \]
  By \cite[Lem.~3.5(3)]{perfect_complexes_stacks},
  $\varphi(\mathcal{T})$ is a Thomason subset of $|X|$. It suffices to
  prove that $\mathcal{I}_{\varphi(\mathcal{T})} = \mathcal{T}$ and
  $\varphi(\mathcal{I}_Y) = Y$.

  Obviously, $\mathcal{T} \subseteq
  \mathcal{I}_{\varphi(\mathcal{T})}$. For the reverse inclusion, if
  $P \in \mathcal{I}_{\varphi(\mathcal{T})}$, then $\supph(P)
  \subseteq \cup_{Q \in \mathcal{T}} \supph(Q)$. Since $\supph(P)$ and
  $\supph(Q)$ are constructible for every $Q\in \mathcal{T}$, it
  follows that there is a finite subset $J \subseteq \mathcal{T}$ such
  that
  \[
  \supph(P) \subseteq \cup_{Q \in J} \supph(Q) =
  \supph(\oplus_{Q \in J} Q).
  \]
  By Lemma \ref{L:supports}, $\tensthck{P} \subseteq
  \tensthck{\oplus_{Q\in J} Q} \subseteq \mathcal{T}$. Thus $P \in
  \mathcal{T}$ and $\mathcal{I}_{\varphi(\mathcal{T})} = \mathcal{T}$.

  Obviously, $Y \supseteq \varphi(\mathcal{I}_Y)$. Since $Y$ is
  Thomason, it is expressible as a union $\cup_\alpha Y_\alpha$ such
  that $|X|\setminus Y_\alpha$ is quasi-compact and open. By
  \cite[Thm.~A]{perfect_complexes_stacks}, for every $\alpha$ there is a
  compact complex $Q_\alpha$ with support $Y_\alpha$. It follows that
  if $y\in Y$, then $y\in \supph(Q_\alpha) \subseteq Y$ for some
  $\alpha$. That is, $y\in \varphi(\mathcal{I}_Y)$ so $Y = \varphi(\mathcal{I}_Y)$.
\end{proof}
\section{The Balmer spectrum of a tame stack}
We will prove Theorem \ref{T:balmer_tame} using \cite[Prop.~6.1]{MR2280286}.
\begin{proof}[Proof of Theorem \ref{T:balmer_tame}]
  Let $s\colon
  (|X|,\supph) \to (|\SPB(\PERF(X))|,\sigma_X)$ be the uniquely induced morphism of support data, where $\sigma_X$
  denotes the universal support datum. By Theorem
  \ref{T:thomason_classification}, $(|X|,\supph)$ is classifying and
  by \cite[Cor.~5.6.1~\&~5.7.2]{MR1771927} we know that $|X|$ is
  spectral. By \cite[Prop.~6.1]{MR2280286}, $s$ is a homeomorphism. By
  definition, $\Orb_{\SPB(\PERF(X))}$ is the sheafification of the
  presheaf
  \[
  (i\colon U \subseteq X) \mapsto \End_{\PERF(X)/\ker (i^*) \cap \PERF(X)}(i^*\Orb_X).
  \]
  Since $|X|$ has a basis consisting of quasi-compact open subsets, it
  is sufficient to identify $\End_{\PERF(X)/\ker(i^*)\cap
    \PERF(X)}(i^*\Orb_X)$ when $i$ is a quasi-compact open
  immersion. By \cite[Lem.~6.7(2)]{perfect_complexes_stacks},
  $\ker(i^*)$ is the localizing envelope of a set of
  objects with compact image in $\DQCOH(X)$. By Thomason's
  Localization Theorem (e.g.,
  \cite[Thm.~4.14]{perfect_complexes_stacks} or
  \cite[Thm.~2.1]{MR1191736}), $\PERF(U)$ is the thick closure of
  $\PERF(X)/\ker (i^*) \cap \PERF(X)$. Since there are natural isomorphisms
  \[
  \End_{\PERF(X)/\ker (i^*) \cap \PERF(X)}(i^*\Orb_X)
  \cong \End_{\PERF(U)}(\Orb_U) \cong\End_{\Orb_U}(\Orb_U) = \Gamma(U,\Orb_X),
  \]
  the result follows.
\end{proof}
% Before we prove Theorem \ref{T:balmer_tame_coarse}, we require the following simple lemma.
% \begin{lemma}\label{L:zar_rings}
%   Let $X$ be an algebraic stack and let $p\colon W \to X$ be a faithfully flat morphism that is locally of finite presentation, where $W$ is a scheme. Let $W_1 = W\times_X W$, which is an algebraic space, and let $R \to W_1$ be an \'etale presentation by a scheme. Then $(|X|,\Orb_{X_{\mathrm{Zar}}})$ is the quotient of $W$ by $R$ in the category of locally ringed spaces. 
% \end{lemma}
% \begin{proof}
%   The induced morphism of topological spaces $|p| \colon |W| \to |X|$ is submersive because $p$ is faithfully flat and locally of finite presentation.  Since $R \to W_1$ and $|W_1| \to |W|\times_{|X|} |W|$ are surjective \cite[Prop.~5.4(iv)]{MR1771927}, it follows that the induced morphism $|W|/|R| \to |X|$ is a homeomorphism. By flat descent, there is also an exact sequence of lisse-\'etale sheaves
%   \[
% \xymatrix{0 \ar[r] & \Orb_{X} \ar[r] & p_*\Orb_W \ar[r] & q_*\Orb_R, }
%   \]
%   where $q\colon R \to X$ is the induced morphism. It follows that if $i \colon U \subseteq X$ is an open immersion, then
%   \begin{align*}
%     \End_{\Orb_U}(\Orb_U) &= \Hom_{\Orb_X}(i_!\Orb_U,\Orb_X)\\
%     &= \Hom_{\Orb_X}(i_!\Orb_U,\ker(p_*\Orb_W \to q_*\Orb_R))\\
%     &= \ker( \Hom_{\Orb_X}(i_!\Orb_U,p_*\Orb_W) \to \Hom_{\Orb_X}(i_!\Orb_U,q_*\Orb_R))\\
%     &= \ker((p_*\Orb_W)(U) \to (q_*\Orb_R)(U)).
%   \end{align*}
%   The result follows.
% \end{proof}
\begin{proof}[Proof of Theorem \ref{T:balmer_tame_coarse}]
  Since $X$ has finite inertia, it has separated diagonal. By
  \cite[Thm.~6.12]{MR3084720}, $\pi$ is a separated universal
  homeomorphism, so $X_{\cms}$ is a quasi-compact and quasi-separated
  algebraic space. By \cite[Thm.~6.12]{MR3084720}, the natural map
  $(|X|,\Orb_{X_{\mathrm{Zar}}}) \to
  (|X_{\cms}|,\Orb_{(X_{\cms})_{\mathrm{Zar}}})$ is an isomorphism of
  locally ringed spaces. By Theorem
  \ref{T:balmer_tame}, the result follows. 
\end{proof} 
\appendix
\section{Tame stacks and coarse spaces}\label{A:tamestacks} In this
appendix, we establish some basic results about $\RDERF \QCPSH{\pi}$,
where $\pi \colon X \to X_{\cms}$ is the coarse space of a
quasi-separated algebraic stack $X$ with finite inertia. Our first result,
however, is a useful lemma that characterises the compact objects on a
certain class of algebraic stacks, which includes $BG$ for all finite
groups $G$. This is likely known, though we are unaware of a reference
for this result in the generality required.
\begin{lemma}\label{L:compactbg}
  Let $W$ be an algebraic stack and let $v\colon V \to W$ be a finite
  and faithfully flat morphism of finite presentation, where $V$ is an
  affine scheme. If $M \in \DQCOH(W)^c$, then $M$ is quasi-isomorphic
  to a direct summand of $\RDERF \QCPSH{v}P$ for some $P\in \PERF(V)$. 
\end{lemma}
\begin{proof}
  If $P \in \PERF(V)$, then $\RDERF\QCPSH{v}P \in \DQCOH(W)^c$
  \cite[Prop.~3.3 \& Ex.~4.8]{perfect_complexes_stacks}. Thus, let
  $\mathcal{T} \subseteq \DQCOH(W)^c$ be the subcategory with objects
  those $N \in \DQCOH(W)^c$ that are quasi-isomorphic to direct summands of $\RDERF\QCPSH{v}P$ for
  some $P\in \PERF(V)$. Clearly, $\mathcal{T}$ is closed under shifts
  and direct summands. We now prove that $\mathcal{T}$ is
  triangulated. Thus let $f\colon N' \to N$ be a morphism in
  $\mathcal{T}$ and complete it to a distinguished triangle 
  \[
  \xymatrix{N' \ar[r]^f & N \ar[r]^c & N'' \ar[r]^{\partial} & N'[1].  }
  \]
  We now prove that $N''\in \mathcal{T}$. By assumption, there are $P$, $P' \in
  \PERF(V)$ and $C$, $C' \in \DQCOH(W)^c$ and quasi-isomorphisms
  $N\oplus C \simeq \RDERF\QCPSH{v}P$, $N'\oplus C' \simeq \RDERF\QCPSH{v}P'$. It follows that
  there is a distinguished triangle
  \[
  \xymatrix{N' \oplus C' \ar[r]^{f\oplus 0} & N \oplus C
    \ar[rr]^-{c\oplus \mathrm{id}_C \oplus 0} & & N'' \oplus C
    \oplus C'[1] \ar[r]^-{\partial \oplus p_{C'[1]}} & N'\oplus C'[1],}
  \]
  where $p_{C'[1]} \colon C\oplus C'[1] \to C'[1]$ is the natural
  projection. In particular, we are reduced to the situation where
  $N'=\RDERF\QCPSH{v}P'$ and $N=\RDERF\QCPSH{v}P$. In this case, the morphism $f\colon N' \to
  N$ by duality induces a morphism $\tilde{f} \colon P' \to v^\times
  \RDERF\QCPSH{v}P$. It follows that the composition $\RDERF\QCPSH{v}P' \xrightarrow{f} \RDERF\QCPSH{v}P
  \to \RDERF\QCPSH{v}v^\times \RDERF \QCPSH{v}P$ is the map $\RDERF\QCPSH{v}\tilde{f}$. Now form a
  distinguished triangle
  \[
  \xymatrix{ P' \ar[r]^-{\tilde{f}} & v^\times \RDERF\QCPSH{v}P \ar[r]^-k & K \ar[r]^-{\delta}
  & P'[1].}
  \]
  Since the morphism $\RDERF\QCPSH{v}P
  \to \RDERF\QCPSH{v}v^\times \RDERF \QCPSH{v}P$ admits a retraction,
  there exists a $Q\in\DQCOH(W)^c$ and a quasi-isomorphism
  $\RDERF\QCPSH{v}v^\times \RDERF \QCPSH{v}P \simeq \RDERF\QCPSH{v}P
  \oplus Q$. There is an induced morphism of distinguished triangles
  \[
  \xymatrix@C3pc{\RDERF\QCPSH{v}P' \ar@{=}[d] \ar[r]^-{\RDERF \QCPSH{v}\tilde{f}} &
    \RDERF\QCPSH{v}v^\times \RDERF \QCPSH{v}P \ar[r]^-{\RDERF \QCPSH{v}k} \ar[d]_\sim & \RDERF
    \QCPSH{v}K \ar[d] \ar[r]^-{\RDERF \QCPSH{v}\delta} &
    \RDERF\QCPSH{v}P'[1] \ar[d] \\ \RDERF\QCPSH{v}P' \ar[r]^{f \oplus 0} & \RDERF \QCPSH{v}P \oplus Q
  \ar[r]^{c\oplus \mathrm{id}_Q} & N'' \oplus Q \ar[r]^-{\partial + 0} & \RDERF\QCPSH{v}P'[1].}
  \]
  It follows that $\RDERF \QCPSH{v}K \simeq N'' \oplus Q$ and so $N''
  \in \mathcal{T}$. 
  By \cite[Ex.~6.5 \& Prop.~6.6]{perfect_complexes_stacks},
  $\DQCOH(W)$ is compactly generated by $v_*\Orb_V$. But Thomason's
  Theorem \cite[Thm.~2.1]{MR1191736} implies that $\DQCOH(W)^c$
  is the smallest thick subcategory containing $v_*\Orb_V$. The result
  follows.  
\end{proof}
The following result was suggested to us by David Rydh. 
\begin{theorem}\label{T:texact_compacts}
If $X$ be a quasi-separated algebraic stack with finite inertia and
coarse space $\pi \colon X \to X_{\cms}$, then the restriction of
$\RDERF \QCPSH{\pi}$ to $\DQCOH(X)^c$ is $t$-exact.  
\end{theorem}
\begin{proof}
  By \cite[Lem.~1.2(4)]{perfect_complexes_stacks}, this may be checked \'etale-locally on $X_{\cms}$. Thus, we may assume that $X_{\cms}$ is an affine scheme. Since $\pi$ is a universal homeomorphism, it follows that $X$ is quasi-compact. Also, since $X$ has finite inertia, it has quasi-finite and separated diagonal. By Theorem \ref{T:qfdiag_pres}, there exist morphisms of algebraic stacks $V \xrightarrow{v} W \xrightarrow{p} X$, such that $V$ is an affine scheme, $v$ is finite faithfully flat and finitely represented and $p$ is a representable, separated and finitely presented Nisnevich covering. By \cite[Prop.~6.5]{MR3084720}, we may further assume that $p$ is fixed-point reflecting. We now apply \cite[Thm.~6.10]{MR3084720} to conclude that the following diagram
\[
 \xymatrix{W \ar[r]^p \ar[d]_{\omega} & X \ar[d]^\pi \\ W_{\cms} \ar[r]^{p_{\cms}} & X_{\cms}}
\]
is cartesian and $p_{\cms}$ is representable, separated, \'etale and of finite presentation. Thus, it suffices to prove the result on $W$. 

Let $M \in \DQCOH(W)^c \cap \DQCOH^{\leq 0}(W)$. By Lemma
\ref{L:compactbg}, we may assume that there is map $i\colon M \to
\RDERF \QCPSH{v}P$, where $P\in \PERF(V)$, that admits a retraction
$r$. It follows that the composition $M \xrightarrow{i}
\RDERF\QCPSH{v}P \to \trunc{>0}\RDERF \QCPSH{v}P$ is the zero
map. Thus the induced map $\RDERF\QCPSH{\omega}M \to \RDERF\QCPSH{\omega}\trunc{>0}\RDERF
\QCPSH{v}P$ is the $0$ map. But $v$ and $\omega\circ v$ are affine, so
there is a natural quasi-isomorphism
$\trunc{>0}\RDERF\QCPSH{\omega}\RDERF\QCPSH{v}P \simeq \RDERF\QCPSH{\omega}\trunc{>0}\RDERF
\QCPSH{v}P$. The resulting map $\trunc{>0}\RDERF\QCPSH{\omega}M \to
\trunc{>0}\RDERF\QCPSH{\omega}\RDERF\QCPSH{v}P$ is $0$ and also
coincides with $\trunc{>0}\RDERF\QCPSH{\omega}(i)$, which admits a
retraction $\trunc{>0}\RDERF\QCPSH{\omega}(r)$. In particular,
$\trunc{>0}\RDERF\QCPSH{\omega}M \simeq 0$ and the result follows.
\end{proof}
In \cite{MR2427954}, they work with a more restrictive definition of
tame, rendering the following corollary a tautology. Indeed, they assume that $X$ has finite inertia and is locally of finite presentation over a base scheme $S$ and that $\pi \colon X \to X_{\cms}$ is such that $\pi_*$ is exact on quasi-coherent sheaves. In our case, we make none of these assumptions, thus it is
non-trivial. 
\begin{corollary}\label{C:tameaov}
  Let $X$ be a quasi-separated algebraic stack with finite inertia
  and coarse space $\pi\colon X \to X_{\cms}$. The following are equivalent:
  \begin{enumerate}
  \item \label{L:tameaov:tame}$X$ is tame,
  \item \label{L:tameaov:ex} $\pi_* \colon \QCOH(X) \to \QCOH(X_{\cms})$ is exact,
  \item \label{L:tameaov:tex-below} $\RDERF \pi_* \colon \DQCOH^+(X)
    \to \DQCOH^+(X_{\cms})$ is $t$-exact, and
  \item \label{L:tameaov:tex-ubd} $\RDERF \QCPSH{\pi} \colon \DQCOH(X)
    \to \DQCOH(X_{\cms})$ is $t$-exact.
  \end{enumerate}
\end{corollary}
\begin{proof}
  We begin with some preliminary reductions. The morphism $\pi$ is a
  separated universal homeomorphism \cite[Thm.~6.12]{MR3084720}, so
  $X_{\cms}$ is a quasi-separated algebraic space and $\pi$ is
  quasi-compact and quasi-separated. Thus by
  \cite[Lem.~1.2(2)]{perfect_complexes_stacks},
  \itemref{L:tameaov:tex-below}$\Rightarrow$\itemref{L:tameaov:tex-ubd}
  and by \cite[Thm.~2.6(2)]{perfect_complexes_stacks} we have   \itemref{L:tameaov:tex-ubd}$\Rightarrow$\itemref{L:tameaov:tex-below}. Clearly,
  \itemref{L:tameaov:tame} may be verified after passing to an affine
  \'etale presentation of $X_{\cms}$, and similarly for
  \itemref{L:tameaov:ex} and \itemref{L:tameaov:tex-below}
  \cite[Lem.~1.2(4) \& Lem.~2.2(6)]{perfect_complexes_stacks}. We may
  consequently assume that $X_{\cms}$ is an affine scheme. Since $\pi$ has
  finite diagonal, it has affine diagonal, so we have
  \itemref{L:tameaov:ex}$\Leftrightarrow$\itemref{L:tameaov:tex-below}
  \cite[Prop.~2.1]{hallj_neeman_dary_no_compacts}. By
  \cite[Thm.~C(1)$\Rightarrow$(3)]{hallj_dary_alg_groups_classifying},
  we now obtain
  \itemref{L:tameaov:ex}$\Rightarrow$\itemref{L:tameaov:tame}. It remains to address
  \itemref{L:tameaov:tame}$\Rightarrow$\itemref{L:tameaov:ex}. 

  Arguing exactly as in the proof of Theorem \ref{T:texact_compacts},
  we may further assume that $X$ admits a finite, faithfully flat and
  finitely presented cover $v\colon V \to X$, where $V$ is an affine
  scheme. Since $X$ is tame, $\Orb_X \in \DQCOH(X)^c$. By Theorem
  \ref{T:texact_compacts}, it follows that the induced morphism $\Orb_X \to
  v_*\Orb_V$ admits a retraction. If $M \in \QCOH(X)$, then it follows
  immediately that the natural map $M \to v_*v^*M$ admits a
  retraction. Thus, if $f\colon M \to N$ is a surjection in
  $\QCOH(X)$, then $f$ is a retraction of the surjection $v_*v^*f$.
  Since $\pi\circ v$ is affine, $\pi_*v_*v^*f$ is surjective. In
  particular, $\pi_*f$ is a retraction of a surjection, thus is
  surjective. The result follows. 
\end{proof}
\section{The induction principle}
The \emph{induction principle} \cite[Tag \spref{08GL}]{stacks-project} for algebraic spaces is closely related to the \'etale d\'evissage results of \cite{MR2774654}. When working with derived categories, where locality results are often quite subtle, it is often advantageous to have the strongest possible criteria at your disposal. In this appendix, we will prove the following induction principle for stacks with quasi-finite and separated diagonal. In \cite{etale_dev_add}, this will be generalized to stacks with non-separated diagonals and put into a broader context. 

Before state this result, we require some notation. Fix an algebraic stack $S$. If $P_1$, $\dots$, $P_r$ is a list of properties of morphisms of algebraic stacks to $S$, let $\STACK_{P_1,\dots,P_r/S}$ denote the full $2$-subcategory of the category of algebraic stacks over $S$ whose objects are those $(x\colon X \to S)$ such that $x$ has properties $P_1$, $\dots$, $P_r$. The following abbreviations will be used: \'et (\'etale), qff (quasi-finite flat), sep~(separated), fp (finitely presented) and rep (representable). 
\begin{theorem}[Induction principle]\label{T:induction_qf_diag}
  Let $S$ be a quasi-compact algebraic stack with quasi-compact and separated diagonal. If $S$ has quasi-finite diagonal, let $\mathbf{E}=\STACK_{\mathrm{rep,sep,qff,fp}/S}$; or if $S$ is Deligne--Mumford, let $\mathbf{E}=\STACK_{\mathrm{rep,sep,\acute{e}t,fp}/S}$. Let $\mathbf{D} \subseteq \mathbf{E}$ be a full subcategory satisfying
  the following properties:
  \begin{enumerate}
  \item[(I1)] if $(X' \to X) \in \mathbf{E}$ is an open immersion and $X \in
    \mathbf{D}$, then $X' \in \mathbf{D}$;
  \item[(I2)] if $(X' \to X) \in \mathbf{E}$ is finite and surjective,
    where $X'$ is an affine scheme, then $X \in \mathbf{D}$; and
  \item[(I3)] if $(U \xrightarrow{i} X)$, $(X'\xrightarrow{f} X) \in
    \mathbf{E}$, where $i$ is an open immersion and $f$ is \'etale and an isomorphism
    over $X\setminus U$, then $X \in \mathbf{D}$ whenever $U$,
    $X'\in\mathbf{D}$. 
  \end{enumerate}
  Then $\mathbf{D} = \mathbf{E}$. In particular, $S \in \mathbf{D}$.
\end{theorem}
\begin{proof}
  Combine Lemma \ref{L:nis_dev} with Theorem \ref{T:qfdiag_pres}.
\end{proof}
We wish to point out that Theorem \ref{T:induction_qf_diag} relies on the existence of coarse spaces for stacks with finite inertia (i.e., the Keel--Mori Theorem \cite{MR1432041,MR3084720}).
\subsection{Nisnevich coverings}
It will be useful to consider some variants and refinements of \cite[\S\S 7-8]{MR2922391}.

If $p \colon W \to X$ is a representable morphism of algebraic stacks, then a
\emph{splitting sequence} for $p$ is a sequence of
quasi-compact open immersions
\[
\emptyset=X_0 \subseteq X_1 \subseteq \cdots \subseteq X_{r} = X
\]
such that $p$ restricted to $X_{i}\setminus X_{i-1}$, when given the induced reduced structure, admits a section
for each $i=1$, $\dots$, $r$. In this situation, we say that $p$ has a
splitting sequence of length $r$. An \'etale and representable morphism of algebraic stacks $p\colon W \to X$ is a \emph{Nisnevich covering} if it admits a splitting sequence.
\begin{example}\label{ex:nisnevich_scheme}
  Let $X$ be a quasi-compact and quasi-separated scheme. Then there exists an affine scheme $W$ and a Nisnevich covering $p\colon W \to X$. Indeed, taking $W=\amalg_{i=1}^n U_i$, where the $\{U_i\}$ form a finite affine open covering of $X$ gives the claim. 
\end{example}
The following lemma is proved by a straightforward induction on the length of the splitting sequence.
\begin{lemma}[Nisnevich d\'evissage]\label{L:nis_dev}
  Let $S$ be a quasi-compact and quasi-separated algebraic stack. 
  Let $\mathbf{E}$ be $\STACK_{\mathrm{rep,\acute{e}t,fp}/S}$ or $\STACK_{\mathrm{rep,sep,\acute{e}t,fp}/S}$. Let
  $\mathbf{D} \subseteq \mathbf{E}$ be a full $2$-subcategory with the following properties:
  \begin{enumerate}
  \item[(N1)] if $(X' \to X) \in \mathbf{E}$ is an open immersion and
    $X \in \mathbf{D}$, then $X' \in \mathbf{D}$; and
  \item[(N2)] if $(U \xrightarrow{i} X)$, $(X'\xrightarrow{f} X) \in
    \mathbf{E}$, where $i$ is an open immersion and $f$ is an isomorphism
    over $X\setminus U$, then $X \in \mathbf{D}$ whenever $U$,
    $X'\in\mathbf{D}$.
  \end{enumerate}
  If $p \colon W \to X$ is a Nisnevich covering in $\mathbf{E}$ and $W \in
  \mathbf{D}$, then $X \in \mathbf{D}$.
\end{lemma}
The following lemma will also be useful.
\begin{lemma}\label{L:nis_props}
  Let $p\colon W \to X$ be a Nisnevich covering of algebraic stacks.
  \begin{enumerate}
  \item\label{L:nis_props:pb} If $f \colon X' \to X$ is a morphism of algebraic stacks, then the pull back $p'\colon W' \to X'$ of $p$ along $f$ is a Nisnevich covering.
  \item\label{L:nis_props:comp} Let $w\colon W' \to W$ be a Nisnevich covering of finite presentation. If $p$ is of finite presentation and $X$ is quasi-compact and quasi-separated, then $p\circ w \colon W' \to X$ is a Nisnevich covering.
  \end{enumerate}
\end{lemma}
 % \begin{proof}
%   The assertion \itemref{L:nis_props:pb} is trivial. For \itemref{L:nis_props:comp}, set $\mathbf{E} = \STACK_{\mathrm{rep,sep,\acute{e}t,fp}/X}$. Let $\mathbf{D}\subseteq \mathbf{E}$ be the full subcategory with objects those $(V \to X)$ such that $W'\times_X V \to V$ is a Nisnevich covering.  We will use Nisnevich d\'evissage (Lemma \ref{L:nis_dev}) to prove that $X\in \mathbf{D}$. Clearly, $\mathbf{D}$ satisfies (N1). A short argument also shows that $\mathbf{D}$ satisfies (N2). It remains to prove that $(W \to X) \in \mathbf{D}$. To see this, we simply note that $W'\times_X W \to W\times_X W \to W$ is a Nisnevich covering because $W\times_X W \to W$ admits a section and $W'\times_X W \to W\times_X W$ is a Nisnevich covering. 
% \end{proof}
\subsection{Presentations}
The following theorem refines \cite[Thm.~7.2]{MR2774654} and will be crucial for the proof of Theorem \ref{T:induction_qf_diag}.
\begin{theorem}\label{T:qfdiag_pres}
  Let $X$ be a quasi-compact algebraic stack with quasi-finite
  and separated diagonal. Then there exist morphisms of algebraic stacks
  \[
  V \xrightarrow{v} W \xrightarrow{p} X
  \]
  such that
  \begin{itemize}
  \item $V$ is an affine scheme;
  \item $v$ is finite, flat, surjective and of finite presentation; and
  \item $p$ is a separated Nisnevich covering of finite presentation.
  \end{itemize}
  In addition, if $S$ is a Deligne--Mumford stack, it can be arranged that $v$ is also \'etale.
\end{theorem}
\begin{proof}
  The proof is similar to \cite[Prop.~6.11]{MR3084720} and \cite[Thm.~7.3]{MR2774654}. 

  By \cite[Thm.~7.1]{MR2774654}, there is an affine scheme $U$ and a representable, separated, quasi-finite, flat, surjective and morphism $u\colon U \to X$ of finite presentation. Let $W=\underline{\mathrm{Hilb}}^{\mathrm{open}}_{{U}/{X}} \to X$ be the subfunctor of the relative Hilbert scheme parameterizing open and closed immersions to $U$ over $X$. It follows that $p\colon W \to X$ is \'etale, representable and separated \cite[Cor.~6.2]{MR2821738}. 

  We now prove that $p$ is a Nisnevich covering. To see this, we note that there exists a sequence of quasi-compact open immersions
  \[
  \emptyset=X_0 \subseteq X_1 \subseteq \cdots \subseteq X_{r} =X
  \]
  such that the restriction of $u$ to $Z_i=\red{(X_{i}\setminus X_{i-1})}$ for
  $i=1$, $\dots$, $r$ is finite, flat and finitely presented. By definition of $p\colon W\to X$, it follows immediately that $p\mid_{Z_i}$ admits a section corresponding to $u\mid_{Z_i}$ and so $p$ is a separated Nisnevich covering. 

  Let $v\colon V \to W$ be the universal family, which is finite, flat, surjective and of finite presentation. Also, $V \to U$ is representable, \'etale and separated \cite[Cor.~6.2]{MR2821738}. Suitably shrinking $W$, we obtain a separated Nisnevich covering $p\colon W \to X$ of finite presentation fitting into a $2$-commutative diagram
  \begin{equation}
  \xymatrix{V \ar[r]^q \ar[d]_v & U \ar[d]^u \\ W \ar[r]^p & X,}\label{eq:1}
\end{equation}  
  and $q$ is \'etale, separated and surjective. By Zariski's Main Theorem \cite[Thm.~A.2]{MR1771927}, $q$ is quasi-affine. By \cite[Thm.~5.3]{MR3084720}, $W$ has a coarse space $\pi\colon W \to W_{\cms}$ such that $W_{\cms}$ is a quasi-affine scheme and $\pi\circ v$ is affine. By Example \ref{ex:nisnevich_scheme} and Lemma \ref{L:nis_props}, we may further reduce to the situation where $W_{\cms}$ is an affine scheme. Since $\pi\circ v$ is affine, the result follows.
\end{proof}
\bibliography{references}

\providecommand{\bysame}{\leavevmode\hbox to3em{\hrulefill}\thinspace}
\providecommand{\MR}{\relax\ifhmode\unskip\space\fi MR }
\providecommand{\MRhref}[2]{%
  \href{http://www.ams.org/mathscinet-getitem?mr=#1}{#2}
}
\providecommand{\href}[2]{#2}
\begin{thebibliography}{{Bal}13}

\bibitem[AOV08]{MR2427954}
D.~Abramovich, M.~Olsson, and A.~Vistoli, \emph{Tame stacks in positive
  characteristic}, Ann. Inst. Fourier (Grenoble) \textbf{58} (2008), no.~4,
  1057--1091.

\bibitem[Bal05]{MR2196732}
P.~Balmer, \emph{The spectrum of prime ideals in tensor triangulated
  categories}, J. Reine Angew. Math. \textbf{588} (2005), 149--168.

\bibitem[{Bal}13]{2013arXiv1309.1808B}
P.~{Balmer}, \emph{{Separable extensions in tt-geometry and generalized Quillen
  stratification}}, preprint, September 2013,
  \href{http://arXiv.org/abs/1309.1808}{\mbox{arXiv:1309.1808}}.

\bibitem[BIK11]{MR2846489}
D.~J. Benson, S.~B. Iyengar, and H.~Krause, \emph{Stratifying modular
  representations of finite groups}, Ann. of Math. (2) \textbf{174} (2011),
  no.~3, 1643--1684.

\bibitem[BKS07]{MR2280286}
A.~B. Buan, H.~Krause, and {\O}.~Solberg, \emph{Support varieties: an ideal
  approach}, Homology, Homotopy Appl. \textbf{9} (2007), no.~1, 45--74.

\bibitem[DM12]{MR2927050}
U.~V. Dubey and V.~M. Mallick, \emph{Spectrum of some triangulated categories},
  J. Algebra \textbf{364} (2012), 90--118.

\bibitem[Ela11]{MR2830235}
A.~D. Elagin, \emph{Cohomological descent theory for a morphism of stacks and
  for equivariant derived categories}, Mat. Sb. \textbf{202} (2011), no.~4,
  31--64.

\bibitem[HNR14]{hallj_neeman_dary_no_compacts}
J.~Hall, A.~Neeman, and D.~Rydh, \emph{One positive and two negative results
  for derived categories of algebraic stacks}, preprint, May 2014,
  \href{http://arXiv.org/abs/1405.1888}{\mbox{arXiv:1405.1888}}.

\bibitem[Hop87]{MR932260}
M.~J. Hopkins, \emph{Global methods in homotopy theory}, Homotopy theory
  ({D}urham, 1985), London Math. Soc. Lecture Note Ser., vol. 117, Cambridge
  Univ. Press, Cambridge, 1987, pp.~73--96.

\bibitem[HR14a]{etale_dev_add}
J.~Hall and D.~Rydh, \emph{Addendum: {\'E}tale d\'evissage, descent and
  pushouts of stacks}, 2014, draft available on request.

\bibitem[HR14b]{hallj_dary_alg_groups_classifying}
\bysame, \emph{Algebraic groups and compact generation of their derived
  categories of representations}, Indiana Univ. Math. J. (2014), accepted for
  publication.

\bibitem[HR14c]{perfect_complexes_stacks}
\bysame, \emph{Perfect complexes on algebraic stacks}, preprint, May 2014,
  \href{http://arXiv.org/abs/1405.1887}{\mbox{arXiv:1405.1887}}.

\bibitem[KM97]{MR1432041}
S.~Keel and S.~Mori, \emph{Quotients by groupoids}, Ann. of Math. (2)
  \textbf{145} (1997), no.~1, 193--213.

\bibitem[Knu71]{MR0302647}
D.~Knutson, \emph{Algebraic spaces}, Lecture Notes in Mathematics, Vol. 203,
  Springer-Verlag, Berlin, 1971.

\bibitem[K{\O}12]{MR2922391}
A.~Krishna and P.~A. {\O}stv{\ae}r, \emph{Nisnevich descent for {$K$}-theory of
  {D}eligne-{M}umford stacks}, J. K-Theory \textbf{9} (2012), no.~2, 291--331.

\bibitem[Kri09]{MR2570954}
A.~Krishna, \emph{Perfect complexes on {D}eligne-{M}umford stacks and
  applications}, J. K-Theory \textbf{4} (2009), no.~3, 559--603.

\bibitem[LMB]{MR1771927}
G.~Laumon and L.~Moret-Bailly, \emph{Champs alg\'ebriques}, Ergebnisse der
  Mathematik und ihrer Grenzgebiete. 3. Folge., vol.~39, Springer-Verlag,
  Berlin, 2000.

\bibitem[Nee92a]{MR1174255}
A.~Neeman, \emph{The chromatic tower for {$D(R)$}}, Topology \textbf{31}
  (1992), no.~3, 519--532, With an appendix by Marcel B{\"o}kstedt.

\bibitem[Nee92b]{MR1191736}
\bysame, \emph{The connection between the {$K$}-theory localization theorem of
  {T}homason, {T}robaugh and {Y}ao and the smashing subcategories of
  {B}ousfield and {R}avenel}, Ann. Sci. \'Ecole Norm. Sup. (4) \textbf{25}
  (1992), no.~5, 547--566.

\bibitem[Nee14]{neeman_unramified2014}
\bysame, \emph{Unramified monoids in $\mathsf{D}_{\mathrm{qc}}({X})$},
  September 2014, submitted.

\bibitem[Ryd11a]{MR2774654}
D.~Rydh, \emph{\'{E}tale d\'evissage, descent and pushouts of stacks}, J.
  Algebra \textbf{331} (2011), 194--223.

\bibitem[Ryd11b]{MR2821738}
\bysame, \emph{Representability of {H}ilbert schemes and {H}ilbert stacks of
  points}, Comm. Algebra \textbf{39} (2011), no.~7, 2632--2646.

\bibitem[Ryd13]{MR3084720}
\bysame, \emph{Existence and properties of geometric quotients}, J. Algebraic
  Geom. \textbf{22} (2013), no.~4, 629--669.

\bibitem[Stacks]{stacks-project}
The Stacks~Project Authors, \emph{{\itshape Stacks Project}},
  \url{http://math.columbia.edu/algebraic_geometry/stacks-git}.

\bibitem[Tho97]{MR1436741}
R.~W. Thomason, \emph{The classification of triangulated subcategories},
  Compositio Math. \textbf{105} (1997), no.~1, 1--27.

\end{thebibliography}
\bibliographystyle{bibstyle}
\end{document}